\documentclass[11pt,reqno]{amsart}
\usepackage{amsfonts}
\usepackage{fancyhdr}
\usepackage{times}
\usepackage[T1]{fontenc}
\usepackage{mathrsfs}
\usepackage{latexsym}
\usepackage[dvips]{graphics}
\usepackage{epsfig}
\usepackage{epsf}
\usepackage{hyperref}
\usepackage{flafter}
\usepackage{amsmath,amsfonts,amsthm,amssymb,amscd}
\input amssym.def
\input amssym.tex
\usepackage{color}
\usepackage[parfill]{parskip}
 \usepackage{mathrsfs}
 \usepackage{geometry}
\usepackage{epstopdf}
\usepackage{verbatim}
\usepackage{ytableau}
\usepackage{mathdots}
\usepackage{bbm}

\newtheorem{thm}{Theorem}
\newtheorem{cor}[thm]{Corollary}

\newtheorem{lem}[thm]{Lemma}
\newtheorem{thm*}{Theorem}
\newtheorem*{con*}{Conjecture}
\newtheorem*{lem*}{Lemma}
\newtheorem{prop}[thm]{Proposition}

\def \id{\operatorname{Id}}

\def \D{\Delta}

\def \e{\varepsilon}

\def \Z{\mathbb{Z}}

\def \k{\mathbbm{k}}

\def \id{\operatorname{Id}}

\def \D{\Delta}

\def \e{\varepsilon}

\def \Z{\mathbb{Z}}

\begin{document}

\title[Combinatorial identity and finite dual]{A combinatorial identity and the finite dual of infinite dihedral group algebra}

\author{Fan Ge \ and \ Gongxiang Liu}

\address{Department of Mathematics, William \& Mary, Williamsburg, VA, United States}

\address{Department of Mathematics, Nanjing University, Nanjing 210093, China}

\email{fge@wm.edu}
\email{gxliu@nju.edu.cn}

\thanks{The first author was partially supported by a startup fund and a summer research award from William \& Mary and the second author is supported by NSFC 11722016.}

\date{}
\maketitle
\begin{abstract}
In this note, we find a combinatorial identity which is closely related to the multi-dimensional integral $\gamma_{m}$ in the study of divisor functions (see \eqref{eq gamma}). As an application, we determine the finite dual of the group algebra of infinite dihedral group.
\end{abstract}
	
\section{Introduction}
Throughout, $\k$ is an algebraically closed field of characteristic zero and all vector spaces are $\k$-spaces. Let $m$ and $n$ be two positive integers, and $U=U(m,n)=\{1,2,...,m+n\}$.
\subsection{Motivation.} It is known that it is quite hard to determine the finite dual $H^{\circ}$ of an infinite dimensional Hopf algebra $H$ in general. Of course, the most direct way to get $H^{\circ}$ is by definition. For this, recall that $H^{\circ}$ is the Hopf algebra generated by $f\in H^{\ast}$ which vanish on an ideal $I\subset H$ of finite codimension. This means that we need a description of \emph{all} finite codimensional ideals which is impossible in general. To the authors' knowledge, there are two other ways to get $H^{\circ}$ if $H$ is good enough.  One is applying the well-known Cartier-Konstant-Milnor-Moore's Theorem (\cite{MM}) if $H$ happens to be commutative. The related idea and method were generalized further (see \cite[Chapter 9]{Mo}, \cite{CM}). Another one is applying representation theoretical way if the representation category Rep-$H$ of finite dimensional modules happens to be very nice (see \cite{Ta}).

The class of affine prime Hopf algebras of Gelfand-Kirillov dimension (GK-dimension for short) has been studied very well in the past (\cite{LWZ,WLD,BZ,L1}) and the regular ones were classified at last in \cite{WLD}. An interesting fact is that all affine prime regular Hopf algebras are commutative-by-finite \cite{BC1}, that is, a finite module over a normal commutative Hopf subalgebra. This suggests that we have a chance to get the finite duals of affine prime regular Hopf algebras of GK-dimension one explicitly. As a step to test it, in this note we consider the group algebra $\k \mathbb{D}_{\infty}$ of the infinite dihedral group $\mathbb{D}_{\infty}$.

Meanwhile, in the study of divisor functions~\cite{KRRR, RS}, the following multi-dimensional integral plays an important role:
	\begin{align}\label{eq gamma}	\gamma_m(c)=\frac{1}{m!G(m+1)^2}\int_{[0,1]^m}\delta(s_1+s_2+\cdots+s_m-c)
\prod_{i<j}(s_i-s_j)^2\ ds_1\cdots ds_m.
	\end{align}
Here $\delta(\cdot)$ is the Dirac $\delta$-function, $m$ is a positive integer and $G(m+1):=(m-1)!\cdots 1!$ is the Barnes G-function. It is shown in~\cite{KRRR} that for each $m$, $\gamma_m(c)$ is a polynomial in $c$. A closely related type of polynomials (in fact, the integral of $\gamma_m$) show up in the study of moments of the Riemann zeta function~\cite{BC}. The polynomials $\gamma_m$ also find connections to Hankel determinants~\cite{BGR, DHI} and the Painlev\'e V equation~\cite{BGR}. Based on these, it is interesting to study the discrete analogue (or some variations) of $\gamma_m$. It turns out that there is a close connection between above two aspects.	
\subsection{Idea and Main results.} To describe the main results, we fix some notions at first. Recall that $m$ and $n$ are positive integers, and $U=U(m,n)=\{1,2,...,m+n\}$. For a set $X=\{x_1,...,x_m\}$ of nonnegative integers whose elements are listed in increasing order, we denote by $V_X$ the Vandermonde determinant of $X$. That is,
$$V_X=\prod_{1\le i<j\le m} \left(x_j-x_i\right).$$	 Our first result is the following identity of combinatorial nature.
	
\begin{thm}\label{thm sum}
Let $t \in \left[\frac{m(m+1)}{2}, \frac{m(m+1)}{2}+mn\right]$ be an integer, and let $t^*=t-\frac{m(m+1)}{2}$. We have
$$\sum_{\substack{X=\{x_1, ..., x_m\}\subset U\\\sum x_i = t }} V_X V_Y = G(m+1) G(n+1)\binom{mn}{t^*}.$$
Here the sum is over all subsets $X$ of $U$ whose elements' sum is $t$, and $Y=U-X$.
\end{thm}

The sum $\sum V_X V_Y$ in our Theorem~\ref{thm sum} should be compared with
\begin{align}\label{eq dis gamma}
	\sum_{\substack{X=\{x_1, ..., x_m\}\subset U\\\sum x_i = t }} V_X^2
\end{align}
which is a discrete analogue of $\gamma_m\cdot G(m+1)^2$. Here we have dropped the $m!$ since each set of values for the $s_i$'s is counted $m!$ times in~\eqref{eq gamma}. To see that~\eqref{eq dis gamma} is indeed a discrete analogue of $\gamma_m\cdot G(m+1)^2$, one can use a Riemann-sum to approximate $\gamma_m$ by restricting each $s_i$ on the set $\{\frac{1}{N}, \frac{2}{N}, ..., \frac{N}{N} \}$ for some large $N$.

The precise expressions for the polynomials $\gamma_m$ are rather complicated (see~\cite{KRRR,BGR}). Thus, one expects a similar complication for its disrecte analogue~\eqref{eq dis gamma}. This is indeed the case. Using Lemma~\ref{lem vandermonde to array} (see below) we can relate the quantity $V_X$ to a counting problem of certain \textit{triangular} arrays, and consequently relate $\sum V_X^2$ to counting certain \textit{square} arrays. The later has been studied in~\cite{KRRR}, and it turns out that there is no simple expression for the value of~\eqref{eq dis gamma}, except for $t$ in a small restricted range. On the contrary, the sum $\sum V_X V_Y$ in Theorem~\ref{thm sum} cannot be related to counting square arrays. Instead, it is about counting pairs of certain triangular arrays. Such problems should in general be more difficult; however, it turns out that in our case we are able to take advantage of the fact that $X$ and $Y$ are complementary sets. This complementary structure makes the corresponding arrays of $X$ and $Y$ relate in a nice way, which enables us to make use of the \textit{Robinson-Schensted-Knuth correspondence} and get a \emph{neat} result on the right-hand side of Theorem~\ref{thm sum}.

By taking $m=n$ in Theorem~\ref{thm sum} we can prove the following result, which relates the combinatorial identity we found to the question about the finite dual $(\k \mathbb{D}_{\infty})^{\circ}$ of the infinite dihedral group algebra $\k\mathbb{D}_{\infty}$.
\begin{cor}\label{cor}
	Let $x$ be an indeterminate and $A$ be the $2m \times 2m$ matrix
	$$\left(
	\begin{array}{cccccccc}
	1 & 0 & \cdots & 0 & 1 & 0 & \cdots & 0 \\
	x & x & \cdots & x & x^{-1} & x^{-1} & \cdots & x^{-1} \\
	x^2 & 2x^2 & \cdots & 2^{m-1}x^2 & x^{-2} & 2x^{-2} & \cdots & 2^{m-1}x^{-2} \\
	x^3 & 3x^3 & \cdots & 3^{m-1}x^3 & x^{-3} & 3x^{-3} & \cdots & 3^{m-1}x^{-3} \\
	\vdots & \vdots & \cdots & \vdots & \vdots & \vdots & \cdots & \vdots \\
	x^{M} & Mx^{M} & \cdots & M^{m-1}x^{M} & x^{-M} & Mx^{-M} & \cdots & M^{m-1}x^{-M}\\
	\end{array}
	\right)$$
	where $M=2m-1$.
	Then the determinant
		\begin{align}\label{eq determinant}
			|A|=G(m+1)^2 \cdot \left(x^{-1}-x\right)^{m^2}.
		\end{align}
\end{cor}
\begin{proof} Multiplying the matrix $A$ by $x,x^2,\ldots,x^{2m-1}$ on the second row, third row,$\ldots$, and the last row respectively, we get
$$B:=\left(
              \begin{array}{cccccccc}
                1 & 0 & \cdots & 0 & 1 & 0 & \cdots & 0 \\
                x^2 & x^2 & \cdots & x^2 & 1 & 1 & \cdots & 1 \\
                x^4 & 2x^4 & \cdots & 2^{m-1}x^4 & 1 & 2 & \cdots & 2^{m-1} \\
                x^6 & 3x^6 & \cdots & 3^{m-1}x^6 & 1 & 3 & \cdots & 3^{m-1} \\
                \vdots & \vdots & \cdots & \vdots & \vdots & \vdots & \cdots & \vdots \\
                x^{2M} & Mx^{2M} & \cdots & M^{m-1}x^{2M} & 1& M & \cdots & M^{m-1}\\
              \end{array}
            \right).$$
To show the result, it is enough to prove that
$$|B|=G(m+1)^2 \cdot x^{m^2-m}(1-x^2)^{m^2}.$$
By definition, $|B|$ is a polynomial of $y:=x^2$.  It is not hard to see that the coefficient of $y^{t}$ (for $\frac{m(m+1)}{2}\leq t\leq \frac{m(m+1)}{2}+m^2$) on left-hand side is exactly equal to $$(-1)^{t-\frac{m(m+1)}{2}}\sum_{\substack{X=\{x_1, ..., x_m\}\subset U\\\sum x_i = t }} V_X V_Y $$ for $n=m$ and on right-hand side is $$(-1)^{t-\frac{m(m+1)}{2}}G(m+1)^2 \binom{m^2}{t^*}.$$ Then we get the desired result by applying Theorem \ref{thm sum} directly. \end{proof}

Now we state our last result. It is known that the group algebra $\k \mathbb{D}_{\infty}$ is a prime regular Hopf algebra of GK-dimension one (see, say \cite{LWZ}). Our goal is to describe the finite dual $\k \mathbb{D}_{\infty}^{\circ}$ of $\k \mathbb{D}_{\infty}$. To this end, we define a  Hopf algebra $\k \mathbb{D}_{\infty^{\circ}}$ using generators and relations. See Subsection \ref{sub3.1} for the definition of this Hopf algebra.

Our final result is
\begin{thm}\label{thm dual} We have a Hopf isomorphism
	$$(\k \mathbb{D}_{\infty})^{\circ}\cong \k \mathbb{D}_{\infty^{\circ}}.$$
\end{thm}
\subsection{Organization and a remark} The note is organized in a quite simple way: In Section 2 we give the proof of Theorem \ref{thm sum} and the Section 3 is devoted to prove the Theorem \ref{thm dual}. We want to say that in a forthcoming work~\cite{L} the second author investigates the finite dual of other types of prime regular Hopf algebras of GK-dimension one.


\section{Proof of Theorem \ref{thm sum}}
In this section we prove Theorem \ref{thm sum}.

\subsection{Some Lemmas.}
\begin{lem}\label{lem vandermonde sum}
Let $A=\{a_1, ..., a_k\}$ be a set of integers. We have
\begin{align}\label{eq vandermonde sum}
	\sum_{\substack{B=\{b_1,...,b_{k-1}\}\subset \mathbb Z\\a_1<b_1\le a_2<b_2\le \cdots <b_{k-1}\le a_k}} V_B = \frac{V_A}{(k-1)!}.
\end{align}
\end{lem}
\proof This is a discrete analogue of Lemma 4.6 in~\cite{KRRR}. To proceed, we write
$$V_B=\det\begin{pmatrix}
1 & 1 & \cdots & 1\\
b_1 & b_2 & \cdots & b_{k-1}\\
\vdots & \vdots & \cdots & \vdots \\
b_1^{k-2} & b_2^{k-2} & \cdots & b_{k-1}^{k-2}
\end{pmatrix},$$
and sum over $b_1, b_2, ..., b_{k-1}$ as in the left-hand side of~\eqref{eq vandermonde sum}. We then arrive at
$$\det\begin{pmatrix}
F_0(a_2, a_1) & F_0(a_3, a_2) & \cdots & F_0(a_k, a_{k-1})\\
F_1(a_2, a_1) & F_1(a_3, a_2) & \cdots & F_1(a_k, a_{k-1})\\
\vdots & \vdots & \cdots & \vdots \\
F_{k-2}(a_2, a_1) & F_{k-2}(a_3, a_2) & \cdots & F_{k-2}(a_k, a_{k-1})
\end{pmatrix},$$
where $F_j(b,a)=\sum_{a<c\le b} c^j = F_j(b,0)-F_j(a,0)$. It is well known that $F_j(b,0)$ is a polynomial in $b$, whose leading term is $\frac{b^{j+1}}{j+1}$. We observe that in the above determinant only the leading term of each $F_j$ makes contribution, while all other terms cancel out by elementary row operations. We thus obtain
\begin{align*}
\det\begin{pmatrix}
a_2-a_1 & a_3- a_2 & \cdots & a_k- a_{k-1}\\
\frac{a_2^2}{2} - \frac{a_1^2}{2} & \frac{a_3^2}{2} - \frac{a_2^2}{2} & \cdots & \frac{a_k^2}{2} - \frac{a_{k-1}^2}{2}\\
\vdots & \vdots & \cdots & \vdots \\
\frac{a_2^{k-1}}{k-1} - \frac{a_1^{k-1}}{k-1} & \frac{a_3^{k-1}}{k-1} - \frac{a_2^{k-1}}{k-1}& \cdots & \frac{a_k^{k-1}}{k-1} - \frac{a_{k-1}^{k-1}}{k-1}
\end{pmatrix}.
\end{align*}
We then write each column as a difference of two column vectors, and expand the determinant as a sum of determinants. It follows that this is exactly an expansion of the following determinant according to the first row
\begin{align*}
\det\begin{pmatrix}
1& 1 & \cdots & 1 \\
a_1 & a_2 & \cdots & a_k\\
\frac{a_1^2}{2} &  \frac{a_2^2}{2} & \cdots & \frac{a_k^2}{2}\\
\vdots & \vdots & \cdots & \vdots \\
\frac{a_1^{k-1}}{k-1} &  \frac{a_2^{k-1}}{k-1}& \cdots & \frac{a_k^{k-1}}{k-1}
\end{pmatrix}.
\end{align*}
Clearly, this is the right-hand side of~\eqref{eq vandermonde sum}. \qed

\begin{lem}\label{lem vandermonde to array}
Let	$X=\{x_1,...,x_m\}$ be a set of positive integers, and $V_X$ be the Vandermonde determinant of $X$. Then  $\frac{V_X}{G(m+1)}$ is equal to the number of triangular arrays satisfying the following.
\begin{itemize}
	\item Each array $A$ is of the form \; $\begin{matrix}
	* & * & \cdots & * & x_m \\
	* & \cdots & * & x_{m-1} \\
	\vdots &  & \iddots \\
	* & x_2 \\
	x_1	
	\end{matrix}$
	\item In each row, we have $\le$ from left to right.
	\item In each column, we have $>$ from top to bottom.
	\item Each entry is a positive integer.
\end{itemize}
\end{lem}

Before we prove Lemma~\ref{lem vandermonde to array}, let us give an example as an illumination. Let $X=\{1, 3, 4\}$, so $m=3$, $G(m+1)=2,$ and $V_X=6.$ We thus have $\frac{V_X}{G(m+1)}=3.$ On the other hand, the number of arrays in the lemma is also $3$, and they are:
$$\begin{matrix}
3 & 4 & 4 \\
2 & 3 \\
1	
\end{matrix} \qquad \qquad \begin{matrix}
4 & 4 & 4 \\
2 & 3 \\
1	
\end{matrix} \qquad \qquad \begin{matrix}
4 & 4 & 4 \\
3 & 3 \\
1	
\end{matrix}$$

\proof  This is also a discrete analogue of a result in~\cite{KRRR}. The proof is straightforward once Lemma~\ref{lem vandermonde sum} is at hand. Denoting a general such array by $$\begin{matrix}
a_{1,1} & a_{1,2} & \cdots & a_{1,m-1} & x_m \\
a_{2,1} & \cdots & a_{2,m-2} & x_{m-1} \\
\vdots &  & \iddots \\
a_{m-1,1} & x_2 \\
x_1	
\end{matrix}$$ we may write the number of such arrays as a multiple sum
$$
\sum_{\substack{x_1<a_{m-1,1}\le x_2\\x_2<a_{m-2,2}\le x_3\\\cdots\\x_{m-1}<a_{1,m-1} \le x_m }} \cdots \sum_{\substack{a_{2,1}\,,a_{1,2}\\a_{3,1}<a_{2,1}\le a_{2,2}\\a_{2,2}<a_{1,2}\le a_{1,3}}}\  \sum_{\substack{a_{1,1}\\a_{2,1}<a_{1,1}\le a_{1,2}}} 1.
$$
We then evaluate the innermost single sum and see that it is the Vandermonde of $\{a_{2,1}\,,a_{1,2}\}$. After that, we evaluate the double sum over $a_{2,1}$ and $a_{1,2}$ using Lemma~\ref{lem vandermonde sum}. The pattern clearly continues. \qed

To state our next lemma, we recall that a \textit{semi-standard Young tableau} (SSYT) is a Young diagram such that in each row we have $\le$ from left to right, and in each column $<$ from top to bottom. Moreover, we say a Young diagram is of \textit{shape} $(\lambda_1, \lambda_2, ..., \lambda_k)$ for some non-negative integers $\lambda_1\ge \lambda_2\ge \cdots \ge \lambda_k$, if it has $\lambda_i$ boxes in the $i$-th row for each $i$. For example, a Young diagram of shape $(5,3, 3)$ looks like
$$
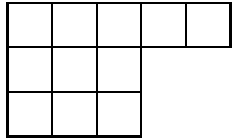
\begin{ytableau} ~& ~& ~& ~&~\\~&~&~\\~&~&~\end{ytableau} $$

\begin{lem}\label{lem array to ssyt}
Let $X\subset U$. Then $\frac{V_X}{G(m+1)}$  is equal to the number of SSYT satisfying
\begin{itemize}
	\item The shape of the SSYTs is $(\tilde{x}_m, \tilde{x}_{m-1},...,\tilde{x}_1 )$, where $\tilde{x_i}=x_i-i$.
	\item The entries in each SSYT belong to the range $\{1,2,...,m\}$.
\end{itemize}
\end{lem}

Note that such SSYT has size $\sum \tilde{x_i} = t^*$.

\proof In view of Lemma~\ref{lem vandermonde to array}, we only need to build a bijection between the arrays in Lemma~\ref{lem vandermonde to array} and the SSYTs in Lemma~\ref{lem array to ssyt}. It is done in two steps.

First, for a given array $A$, we subtract $k$ from each entry in the $k$-th row counting from the bottom (do it for every row); and then we shift it and fill in $0$'s to make it a square array, say $\tilde{A}$, as follows.
$$\begin{matrix}
* & * & * & \cdots & * & \tilde{x}_m \\
0 & * & * & \cdots & * & \tilde{x}_{m-1} \\
0 & 0 &  * & \cdots & * & \tilde{x}_{m-2}\\
\vdots & \vdots & \cdots & & \vdots &\vdots\\
0 & 0 & \cdots & 0 & * & \tilde{x}_2 \\
0 & 0 & \cdots & 0 & 0 & \tilde{x}_1	
\end{matrix}$$
Note that such arrays satisfy
\begin{itemize}
	\item In each row, we have $\le$ from left to right.
\item In each column, we have $\ge$ from top to bottom.
\item Each entry is a non-negative integer.
\end{itemize}

Next, we use an idea in~\cite{KRRR} to build a bijection between such arrays $\tilde{A}$ with the said SSYTs. The idea is that the value of the $(i,j)$-th entry in an array $\tilde{A}$ corresponds to the rightmost place of the number $j$ appearing in the $i$-th row of the corresponding SSYT; if the $(i,j-1)$-th entry and the $(i,j)$-th entry are the same in an array, then $j$ does not appear in the $i$-th row of the SSYT. Thus, the fact that the entries in SSYTs live in $\{1,2,...,m\}$ agrees with the fact that $\tilde{A}$ has $m$ columns. Moreover, the last column of $\tilde{A}$ agrees with the shape of the SSYTs. Below is an example of this correspondence.
$$A: \quad \begin{matrix}
5 & 7 & 7  \\
3 & 5  \\
 1\\
\end{matrix} \qquad \qquad\tilde{A}:\quad \begin{matrix}
2 & 4 & 4  \\
0 & 1 & 3 \\
0 & 0 &  0\\
\end{matrix} \qquad \qquad \text{SSYT}: \quad\begin{ytableau} 1&1&2&2\\2&3&3\end{ytableau} $$
In this example, in $\tilde{A}$ the first row first column (i.e., the (1,1) entry) is $2$; this means that in the first row of the SSYT, the rightmost position of $1$ is the second box, and thus we have $1$ in the first two  boxes. The $(1,2)$ entry is $4$ in $\tilde{A}$; that means in the first row of the SSYT, the rightmost position of $2$ is the fourth box, and so we fill in $2$ till the fourth box.

It is straightforward to verify that this correspondence between the arrays and the SSYTs is a bijection. \qed

\begin{lem}\label{lem complement}
	Let $Y=\{y_1,...,y_n\}\subset U$. Then $\frac{V_Y}{G(n+1)}$ is the number of SSYTs satisfying
	\begin{itemize}
		\item The shape of the SSYTs is $(m+1-y_1, \ m+2-y_2, ..., \ m+n-y_n)$.
		\item The entries in each SSYT belong to the range $\{1,2,...,n\}$.
	\end{itemize}
\end{lem}
\proof We first observe that $V_Y$ is equal to the Vandermonde determinant of the set
$$\{m+n+1-y_n,\ m+n+1-y_{n-1}, ...,\ m+n+1-y_1\}.$$
The result then follows from Lemma~\ref{lem array to ssyt}. \qed

We also require the concept of a \textit{transpose shape}, which has the same spirit as a transpose of a matrix; namely, the $i$-th row becomes the $i$-th column for each $i$. For example, the transpose shape of $(5,3,3)$ is $(3,3,3,1,1)$:
$$
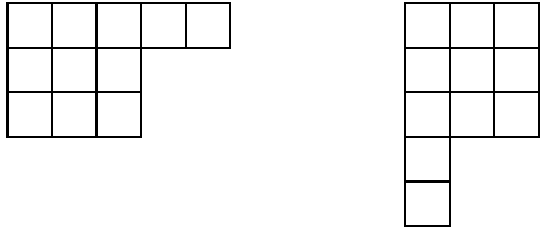
\begin{ytableau} ~& ~& ~& ~&~\\~&~&~\\~&~&~\end{ytableau} \qquad \qquad \qquad \begin{ytableau} ~& ~& ~\\~&~&~\\~&~&~ \\~ \\~\end{ytableau} $$

\begin{lem}\label{lem transpose}
	Let $X=\{x_1,...,x_m\}$  and $Y=\{y_1,...,y_n\}$, where $X\cup Y=U$. Then the shape $(m+1-y_1, \ m+2-y_2, ..., \ m+n-y_n)$ is the transpose of the shape $(\tilde{x}_m, \tilde{x}_{m-1},...,\tilde{x}_1 )$, where $\tilde{x_i}=x_i-i$ as in Lemma~\ref{lem array to ssyt}.
\end{lem}

\proof We prove it by induction on $m+n$. The lemma is clearly true for $m=n=1$. Now assume it is true for all pairs of positive integers $(i,j)$ with $i+j\le k$. We shall prove that the statement holds true for all pairs $(m,n)$ with $m+n=k+1$. We take an arbitrary such pair $(m,n)$, and keep in mind that that $X$ is a subset of $\{1,2,...,k+1\}$ with size $m$, and $Y$ is the complement of $X$.

We first assume that $k+1\in X$, and separate into two cases.

\textit{Case 1.} $m=1$. That is, $X=\{k+1\}$ and $Y=\{1,2,...,k\}$. In this case we do not even need the induction hypothesis. One simply verify that the shape $(m+1-y_1, \ m+2-y_2, ..., \ m+n-y_n)$ (here $m=1, n=k, y_i=i$ for all $i$) is just $(1,1,...,1)$ with $k$ components. On the other hand, the shape $(\tilde{x}_m, \tilde{x}_{m-1},...,\tilde{x}_1 )$ is just $(k)$. Clearly they are transpose to each other.

\textit{Case 2.} $m\ge 2$. So $x_{m}=k+1.$ For convenience we write $m'=m-1$ and $X'=\{x_1,..., x_{m'}\}$. Since $X'\cup Y=\{1,2,...,k\}$, we may use the induction hypothesis and conclude that the transpose of the shape $(\tilde{x}_{m'},...,\tilde{x}_1 )$  is the shape $(m'+1-y_1, \ m'+2-y_2, ..., \ m'+n-y_n)$. We need to show that
the transpose of $(\tilde{x}_{m}, \ \tilde{x}_{m'}, ..., \ \tilde{x}_1 )$  is $(m+1-y_1, \ m+2-y_2, ..., \ m+n-y_n)$.

Since $(\tilde{x}_{m}, \ \tilde{x}_{m'}, ..., \ \tilde{x}_1 ) = (k+1-m,  \ \tilde{x}_{m'}, ..., \ \tilde{x}_1 ) = (n,  \ \tilde{x}_{m'}, ..., \ \tilde{x}_1 )$, its transpose should be obtained by modifying the transpose $T$ of $(\tilde{x}_{m'},...,\tilde{x}_1 )$ by adding $1$ to each component of $T$. But here we need to be careful, as a shape is not uniquely determined by a tuple of integers, the reason being that one could add an arbitrary number of $0$'s to the tail of a tuple and does not change the shape. Therefore, when we choose a tuple for $T$ and add $1$ to its each component, at last we need to make sure that the resulting shape has the number of rows the same as the number of columns of  $(n, \tilde{x}_{m'},...,\tilde{x}_1 )$. This number is clearly $n$. We then verify that by adding $1$ to each component of $(m'+1-y_1, \ m'+2-y_2, ..., \ m'+n-y_n)$ (this is a tuple for the transpose $T$ of $(\tilde{x}_{m'},...,\tilde{x}_1 )$, by induction hypothesis) we obtain $(m+1-y_1, \ m+2-y_2, ..., \ m+n-y_n)$, whose number of rows is exactly $n$, as $m+n-y_n\ge 1$. So indeed we see that this shape is the transpose of $(\tilde{x}_{m}, \ \tilde{x}_{m'}, ..., \ \tilde{x}_1 )$.

If instead $k+1\in Y$, the argument is similar (and in fact, easier) and we omit the details. \qed


\subsection{Proof of Theorem~\ref{thm sum}}
Collecting Lemmas~\ref{lem array to ssyt}, \ref{lem complement} and~\ref{lem transpose}, we see that $\frac{V_X}{G(m+1)}\frac{V_Y}{G(n+1)}$ is the number of pairs of SSYTs $(P,Q)$ satisfying the following.
\begin{itemize}
	\item The shape of $P$ is $(\tilde{x}_m, \tilde{x}_{m-1},...,\tilde{x}_1 )$.
	\item The shape of $Q$ is the shape of the transpose of $P$.
	\item The entries in $P$ belong to the range $\{1,2,...,m\}$.
	\item The entries in $Q$ belong to the range $\{1,2,...,n\}$.
\end{itemize}
Note that these SSYTs has size $\sum \tilde{x_i} = t^*$.
Thus, $$\sum_{\substack{\{x_1, ..., x_m\}\in U\\\sum x_i = t }} \frac{V_X}{G(m+1)}\frac{V_Y}{G(n+1)}$$ is the number of pairs of SSYTs $(P,Q)$ satisfying the following.
\begin{itemize}
	\item The size of $P$ is $t^*$.
	\item The shape of $Q$ is the shape of the transpose of $P$.
	\item The entries in $P$ belong to the range $\{1,2,...,m\}$.
	\item The entries in $Q$ belong to the range $\{1,2,...,n\}$.
\end{itemize}
But according to the Robinson-Schensted-Knuth correspondence (see, for example, Section 4.3 in~\cite{Kra} or Appendix A in~\cite{Ful}), the number of such pairs is exactly the number of $(0,1)$ matrices of size $n\times m$ that contains $t^*$ many $1$'s. Hence the result. \qed

\section{Proof of Theorem~\ref{thm dual}}

Recall that by definition the infinite dihedral group $\mathbb{D}_{\infty}$ is generated by two elements $g$ and $x$ satisfying
      $$x^2=1,\;\;xgx=g^{-1}.$$
To determine the finite dual of $\k\mathbb{D}_{\infty}$, we define a Hopf algebra using generators and relations at first.
\subsection{The Hopf algebra $\k\mathbb{D}_{\infty^{\circ}}$.}\label{sub3.1} As an algebra, $\k \mathbb{D}_{\infty^{\circ}}$ is generated by $F, \phi_{\lambda}, \psi_{\lambda}$ for $\lambda\in \k^{\ast}=\k\setminus \{0\}$ and subjects to the following relations
\begin{align*} &F\phi_{\lambda}=\phi_{\lambda}F,\;\;F\psi_{\lambda}=\psi_{\lambda}F,\;\;\phi_1=1,\\
&\phi_{\lambda}\psi_{\lambda'}=\psi_{\lambda'}\phi_{\lambda}=\psi_{\lambda\lambda'},\;\;
\phi_{\lambda}\phi_{\lambda'}=\phi_{\lambda\lambda'},\;\;
\psi_{\lambda}\psi_{\lambda'}=\phi_{\lambda\lambda'}
\end{align*}
for all $\lambda,\lambda'\in \k^{\ast}.$
The comultiplication, counit and the antipode are given by
\begin{align*} &\D(F)=F\otimes 1+\psi_1\otimes F,\;\;\D(\phi_{\lambda})=
\frac{1}{2}(\phi_{\lambda}+\psi_{\lambda})\otimes \phi_{\lambda}+\frac{1}{2}(\phi_{\lambda}-\psi_{\lambda})\otimes \phi_{\lambda^{-1}},\;\;\\
&\D(\psi_{\lambda})=
\frac{1}{2}(\phi_{\lambda}+\psi_{\lambda})\otimes \psi_{\lambda}-\frac{1}{2}(\phi_{\lambda}-\psi_{\lambda})\otimes \psi_{\lambda^{-1}},\\
&\e(F)=0,\;\;\e(\phi_{\lambda})=\e(\psi_{\lambda})=1,\\
&S(F)=-\psi_{1}F,\;\;S(\phi_{\lambda})=\frac{1}{2}(\phi_{\lambda^{-1}}+\psi_{\lambda^{-1}})
+\frac{1}{2}(\phi_{\lambda}-\psi_{\lambda}),\;\;\\
&S(\psi_{\lambda})=\frac{1}{2}(\phi_{\lambda^{-1}}+\psi_{\lambda^{-1}})
-\frac{1}{2}(\phi_{\lambda}-\psi_{\lambda})
\end{align*}
for $\lambda\in \k^{\ast}.$

\begin{lem} With operations defined above, $\k\mathbb{D}_{\infty^{\circ}}$ is a Hopf algebra.\end{lem}
\begin{proof} The proof is routine and we omit most of it. The only point we want to show is the coassociativity and the axiom for antipode of these generators. By definition, it is not hard to see that $\phi_{\pm1}$ and $\psi_{\pm1}$ are group-like elements. This implies that $F$ is a $(1,\psi_1)$-skew primitive element. So to show the coassociativity and the axiom for antipode of these generators, it is enough to show that $(\id\otimes \D)\D(\phi_{\lambda})=(\D\otimes \id)\D(\phi_{\lambda}),\;(\id\otimes \D)\D(\psi_{\lambda})=(\D\otimes \id)\D(\psi_{\lambda})$ and $S(\phi_{\lambda}')\phi_{\lambda}''=\phi_{\lambda}'S(\phi_{\lambda}'')
=S(\psi_{\lambda}')\psi_{\lambda}''=\psi_{\lambda}'S(\psi_{\lambda}'')=1$ for $\lambda\in \k^{\ast}$.  For this, let
$$e_{\lambda}:= \frac{1}{2}(\phi_{\lambda}+\psi_{\lambda}),\;\;\;\;f_{\lambda}:= \frac{1}{2}(\phi_{\lambda}-\psi_{\lambda}).$$
Clearly, to show the coassocitivity of $\phi_{\lambda}$ and $\psi_{\lambda}$, it is enough to show that for $e_{\lambda}$ and $f_{\lambda}$ for $\lambda\in \k^{\ast}.$    By the definition of the comultiplication, we have
\begin{align*} \D(e_{\lambda})&=\D(\frac{1}{2}(\phi_{\lambda}+\psi_{\lambda}))\\
&=\frac{1}{2}(\phi_{\lambda}+\psi_{\lambda})\otimes \frac{1}{2}(\phi_{\lambda}+\psi_{\lambda})+\frac{1}{2}(\phi_{\lambda}-\psi_{\lambda})\otimes
\frac{1}{2}(\phi_{\lambda^{-1}}-\psi_{\lambda^{-1}})\\
&=e_{\lambda}\otimes e_{\lambda}+f_{\lambda}\otimes f_{\lambda^{-1}}.
\end{align*}
Similarly, one can show that $\D(f_{\lambda})=e_{\lambda}\otimes f_{\lambda}+f_{\lambda}\otimes e_{\lambda^{-1}}.$ Now, a simple computation shows that
\begin{align*} (\id\otimes\D)\D(e_{\lambda})&=(\id\otimes \D)(e_{\lambda}\otimes e_{\lambda}+f_{\lambda}\otimes f_{\lambda^{-1}})\\
&=e_{\lambda}\otimes (e_{\lambda}\otimes e_{\lambda}+f_{\lambda}\otimes f_{\lambda^{-1}})+f_{\lambda}\otimes(e_{\lambda^{-1}}\otimes f_{\lambda^{-1}}+f_{\lambda^{-1}}\otimes e_{\lambda});\\
(\D\otimes\id)\D(e_{\lambda})&=(\D\otimes \id)(e_{\lambda}\otimes e_{\lambda}+f_{\lambda}\otimes f_{\lambda^{-1}})\\
&=(e_{\lambda}\otimes e_{\lambda}+f_{\lambda}\otimes f_{\lambda^{-1}})\otimes e_{\lambda}+(e_{\lambda}\otimes f_{\lambda}+f_{\lambda}\otimes e_{\lambda^{-1}})\otimes f_{\lambda^{-1}}.
\end{align*}
It is not hard to see that they indeed the same and thus $(\id\otimes\D)\D(e_{\lambda})=(\D\otimes\id)\D(e_{\lambda}).$ Similarly, one can show that both $(\id\otimes\D)\D(f_{\lambda})$ and $(\D\otimes\id)\D(f_{\lambda})$ equal to
$$e_{\lambda}\otimes e_{\lambda}\otimes f_{\lambda}+
e_{\lambda}\otimes f_{\lambda}\otimes e_{\lambda^{-1}}+
f_{\lambda}\otimes e_{\lambda^{-1}}\otimes e_{\lambda^{-1}}+
f_{\lambda}\otimes f_{\lambda^{-1}}\otimes f_{\lambda}.$$

To show the axiom of the antipode, we also just need test it one new generators $e_{\lambda}$ and $f_{\lambda}.$ By definition, we find that
\begin{align*}&e_{\lambda}e_{\lambda'}=e_{\lambda\lambda'},\;\;
f_{\lambda}f_{\lambda'}=f_{\lambda\lambda'},\;\;
e_{\lambda}f_{\lambda'}=f_{\lambda'}e_{\lambda}=0,\;\;\\
&\varepsilon(e_{\lambda})=1,\;\;\varepsilon(f_{\lambda})=0,\\
&S(e_{\lambda})=e_{\lambda^{-1}},\;\;S(f_{\lambda})=f_{\lambda},\end{align*}
for $\lambda,\lambda'\in \k^{\ast}.$
Therefore,
\begin{align*}e_{\lambda}'S(e_{\lambda}'')&=e_{\lambda}e_{\lambda^{-1}}+
f_{\lambda}f_{\lambda^{-1}}=e_{1}+f_{1}=\phi_1=1=\varepsilon(e_{\lambda}),\\
S(e_{\lambda}')e_{\lambda}''&=e_{\lambda^{-1}}e_{\lambda}+
f_{\lambda}f_{\lambda^{-1}}=e_{1}+f_{1}=\phi_1=1=\varepsilon(e_{\lambda}),\\
f_{\lambda}'S(f_{\lambda}'')&=e_{\lambda}f_{\lambda}+
f_{\lambda}e_{\lambda}=0=\varepsilon(f_{\lambda}),\\
S(f_{\lambda}')f_{\lambda}''&=e_{\lambda^{-1}}f_{\lambda}+
f_{\lambda}e_{\lambda^{-1}}=0=\varepsilon(f_{\lambda})
\end{align*}
for $\lambda\in \k^{\ast}.$
\end{proof}

Our strategy to prove Theorem \ref{thm dual} is as follows. Firstly, we construct Hopf subalgebra $A$ of $(\k \mathbb{D}_{\infty})^{\circ}$; Secondly, we show that $A$ is indeed the whole $(\k \mathbb{D}_{\infty})^{\circ}$; At last, we prove that there is a natural isomorphism between $\k \mathbb{D}_{\infty^{\circ}}$ and $A$.

\subsection{Construction of $A$}

 Clearly, $\{g^{i}x^{j}|i\in \Z, j=0,1\}$ is a basis of $\k \mathbb{D}_{\infty}.$ Denote its dual basis by $f_{i,j}$, that is, $f_{i,j}(g^{i'}x^{j'})=\delta_{i,i'}\delta_{j,j'}$ for $i,i'\in \Z,\;j,j'=0,1$ and $\delta_{\cdot,\cdot}$ the Kronecher's function. We construct the following elements in $(\k \mathbb{D}_{\infty})^{\ast}$:
\begin{equation}\label{elements} E:= \sum_{i\in \Z}i(f_{i,0}+f_{i,1}),\;\;\Phi_{\lambda}:= \sum_{i\in \Z}\lambda^i(f_{i,0}+f_{i,1}),\;\;\Psi_{\lambda}:= \sum_{i\in \Z}\lambda^i(f_{i,0}-f_{i,1})
	\end{equation}
	for $\lambda\in \k^{\ast}.$  It is straightforward to check that
	$$E(((g-1)^2))=\Phi_{\lambda}(((g-\lambda)(g-\lambda^{-1})))=
\Psi_{\lambda}(((g-\lambda)(g-\lambda^{-1})))=0,$$
	and therefore, all the  elements in~\eqref{elements} live in the finite dual $(\k \mathbb{D}_{\infty})^{\circ}.$ Here $((g-\lambda)(g-\lambda^{-1}))$	means the ideal generated by $(g-\lambda)(g-\lambda^{-1})$.	By definition,
	\begin{equation}\label{e6.2}\Phi_1|_{\mathbb{D}_\infty}=1\end{equation} and thus, it is the multiplicative identity of $(\k \mathbb{D}_{\infty})^{\circ}$.

Now we define $A$ to be subalgebra of $(\k \mathbb{D}_{\infty})^{\circ}$ which is generated by $E, \Phi_{\lambda}, \Psi_{\lambda}$ for $\lambda\in \k^{\ast}.$

\begin{lem} For the algebra $A$, we have the following equations
	\begin{align}\label{e.6.3}&E\Phi_{\lambda}=\Phi_{\lambda}E,\;\;E\Psi_{\lambda}=\Psi_{\lambda}E,\;\;
\\	&\label{e6.4}\Phi_{\lambda}\Phi_{\lambda'}=\Phi_{\lambda\lambda'},\;\;\Psi_{\lambda}\Psi_{\lambda'}=\Phi_{\lambda\lambda'},\;\;
\Phi_{\lambda}\Psi_{\lambda'}=\Psi_{\lambda\lambda'}
	\end{align}
for $\lambda, \lambda'\in \k^{\ast}.$
\end{lem}
\begin{proof} Since $\k \mathbb{D}_{\infty}$ is cocommutative, $(\k \mathbb{D}_{\infty})^{\circ}$ is a commutative algebra. Therefore, we have the equation \eqref{e.6.3}. Since $$f_{i,0}f_{j,0}=\delta_{i,j}f_{i,0},\;\;
	f_{i,1}f_{j,1}=\delta_{i,j}f_{i,1},\;\;f_{i,0}f_{j,1}=0$$
	for $i,j\in \Z$, we have
	\begin{align}&\Phi_{\lambda}\Phi_{\lambda'}=\sum_{i\in \Z}\lambda^i(f_{i,0}+f_{i,1})
	\sum_{j\in \Z}\lambda'^j(f_{j,0}+f_{j,1})=\sum_{i \in \Z}(\lambda\lambda')^i(f_{i,0}+f_{i,1})=\Phi_{\lambda\lambda'},\notag\\
	&\Psi_{\lambda}\Psi_{\lambda'}=\sum_{i\in \Z}\lambda^i(f_{i,0}-f_{i,1})
	\sum_{j\in \Z}\lambda'^j(f_{j,0}-f_{j,1})=\sum_{i \in \Z}(\lambda\lambda')^i(f_{i,0}+f_{i,1})=\Phi_{\lambda\lambda'},\notag\\
	&\Phi_{\lambda}\Psi_{\lambda'}=\Psi_{\lambda'}\Phi_{\lambda}=\sum_{i\in \Z}\lambda^i(f_{i,0}+f_{i,1})
	\sum_{j\in \Z}\lambda'^j(f_{j,0}-f_{j,1})=\sum_{i \in \Z}(\lambda\lambda')^i(f_{i,0}-f_{i,1})=\Psi_{\lambda\lambda'}\notag
	\end{align}
	for $\lambda,\lambda'\in \k^{\ast}.$
\end{proof}
	
\begin{lem}\label{l11} The subalgebra $A$ is a Hopf subalgebra and the actions of comultiplication, counit and the antipode are give by
\begin{align}&\D(E)=E\otimes 1+\Psi_1\otimes E,\;\;\D(\Phi_{\lambda})=
\frac{1}{2}(\Phi_{\lambda}+\Psi_{\lambda})\otimes \Phi_{\lambda}+\frac{1}{2}(\Phi_{\lambda}-\Psi_{\lambda})\otimes \Phi_{\lambda^{-1}},\\
&\D(\Psi_{\lambda})=
\frac{1}{2}(\Phi_{\lambda}+\Psi_{\lambda})\otimes \Psi_{\lambda}-\frac{1}{2}(\Phi_{\lambda}-\Psi_{\lambda})\otimes \Psi_{\lambda^{-1}},\\
&\varepsilon(\Phi_{\lambda})=\varepsilon(\Psi_{\lambda})=1,\;\;\varepsilon(E)=0,\\
& S(\Phi_{\lambda})=\frac{1}{2}(\Phi_{\lambda^{-1}}+\Psi_{\lambda^{-1}})
+\frac{1}{2}(\Phi_{\lambda}-\Psi_{\lambda}),\;\;\\
&S(\Psi_{\lambda})=\frac{1}{2}(\Phi_{\lambda^{-1}}+\Psi_{\lambda^{-1}})
-\frac{1}{2}(\Phi_{\lambda}-\Psi_{\lambda}),\;\;
S(E)=-\Psi_{1}E.
\end{align}
\end{lem}
\begin{proof} It is enough to determine the expression of the comultiplication. All of them based on the following simple computation: $$\D(f_{i,0})=\sum_{j\in \Z}(f_{j,0}\otimes f_{i-j,0}+
f_{j,1}\otimes f_{j-i,1}),\;\;\D(f_{i,1})=\sum_{j\in \Z}(f_{j,0}\otimes f_{i-j,1}+
f_{j,1}\otimes f_{j-i,0}).$$ Actually from this we have
\begin{eqnarray*}\D(E)&=&\D(\sum_{i\in \Z} i(f_{i,0}+f_{i,1}))=\sum_{i\in \Z} i(\D(f_{i,0})+\D(f_{i,1}))\\
&=&\sum_{i\in \Z}i(\sum_{j\in \Z}(f_{j,0}\otimes f_{i-j,0}+
f_{j,1}\otimes f_{j-i,1})+\sum_{j\in \Z}(f_{j,0}\otimes f_{i-j,1}+
f_{j,1}\otimes f_{j-i,0}))\\
&=&\sum_{i,j\in\Z}(jf_{j,0}\otimes f_{i-j,0}+f_{j,0}\otimes (i-j)f_{i-j,0})+
\sum_{i,j\in\Z}(jf_{j,1}\otimes f_{j-i,1}+f_{j,1}\otimes (i-j)f_{j-i,1})\\
&&+\sum_{i,j\in\Z}(jf_{j,0}\otimes f_{i-j,1}+f_{j,0}\otimes (i-j)f_{i-j,1})+
\sum_{i,j\in\Z}(jf_{j,1}\otimes f_{j-i,0}+f_{j,1}\otimes (i-j)f_{j-i,0})\\
&=&\sum_{i,j\in\Z}(jf_{j,0}\otimes f_{i-j,0}+jf_{j,0}\otimes f_{i-j,1})+\sum_{i,j\in\Z}
(f_{j,0}\otimes (i-j)f_{i-j,0}+f_{j,0}\otimes (i-j)f_{i-j,1})\\
&&+\sum_{i,j\in\Z}(jf_{j,1}\otimes f_{j-i,1}+jf_{j,1}\otimes f_{j-i,0})+
\sum_{i,j\in\Z}(f_{j,1}\otimes (i-j)f_{j-i,1}+f_{j,1}\otimes (i-j)f_{j-i,0})\\
&=&\sum_{j\in \Z}jf_{j,0}\otimes \Phi_1+\sum_{j\in\Z}f_{j,0}\otimes E+
\sum_{j\in \Z}jf_{j,1}\otimes \Phi_1-\sum_{j\in\Z}f_{j,1}\otimes E\\
&=&E\otimes 1+\Psi_1\otimes E,
\end{eqnarray*}
and
\begin{eqnarray*}\D(\Phi_{\lambda})&=&\D(\sum_{i\in \Z} \lambda^{i}(f_{i,0}+f_{i,1}))=\sum_{i\in \Z} \lambda^{i}(\D(f_{i,0})+\D(f_{i,1}))\\
&=&\sum_{i\in \Z}\lambda^{i}(\sum_{j\in \Z}(f_{j,0}\otimes f_{i-j,0}+
f_{j,1}\otimes f_{j-i,1})+\sum_{j\in \Z}(f_{j,0}\otimes f_{i-j,1}+
f_{j,1}\otimes f_{j-i,0}))\\
&=&\sum_{i,j\in\Z}(\lambda^{j}f_{j,0}\otimes \lambda^{i-j}f_{i-j,0}+\lambda^{j}f_{j,1}\otimes \lambda^{i-j}f_{j-i,1})\\
&&+\sum_{i,j\in\Z}(\lambda^{j}f_{j,0}\otimes \lambda^{i-j}f_{i-j,1}+\lambda^{j}f_{j,1}\otimes \lambda^{i-j}f_{j-i,0})\\
&=&\sum_{i,j\in\Z}(\lambda^{j}f_{j,0}\otimes \lambda^{i-j}f_{i-j,0}+\lambda^{j}f_{j,0}\otimes \lambda^{i-j}f_{i-j,1})\\
&&+\sum_{i,j\in\Z}(\lambda^{j}f_{j,1}\otimes \lambda^{i-j}f_{j-1,1}+\lambda^{j}f_{j,1}\otimes \lambda^{i-j}f_{j-i,0})\\
&=&\sum_{j\in \Z}\lambda^{j}f_{j,0}\otimes \Phi_{\lambda}+\sum_{j\in \Z}\lambda^{j}f_{j,1}\otimes \Phi_{\lambda^{-1}}\\
&=&\frac{1}{2}(\Phi_{\lambda}+\Psi_{\lambda})\otimes \Phi_{\lambda}+\frac{1}{2}(\Phi_{\lambda}-\Psi_{\lambda})\otimes \Phi_{\lambda^{-1}}.
\end{eqnarray*}
Similarly, one can show that $\D(\Psi_{\lambda})=
\frac{1}{2}(\Phi_{\lambda}+\Psi_{\lambda})\otimes \Psi_{\lambda}-\frac{1}{2}(\Phi_{\lambda}-\Psi_{\lambda})\otimes \Psi_{\lambda^{-1}}.$

By the axioms of the definition of a Hopf algebra, the expressions of counit and the antipode can be gotten easily.
 \end{proof}	
\subsection{$A=(\k\mathbb{D}_{\infty})^{\circ}$. } 	For the reader's convenience, we formulate this fact as a proposition. We remark that this is the place where we need Corollary \ref{cor}.
\begin{prop}\label{prop12} As an algebra, $(\k\mathbb{D}_{\infty})^{\circ}$ is generated by $E, \Phi_{\lambda}$ and $\Psi_{\lambda}$ for $\lambda\in \k^{\ast}$, that is,
$$A=(\k\mathbb{D}_{\infty})^{\circ}.$$
\end{prop}
\begin{proof} To show the result, it suffices to show that for an arbitrary cofinite ideal $I\subset \k D_{\infty}$, we have $(\k D_{\infty}/I)^\ast\subset A.$ 	We divide this into several steps.\\[1.5mm]
\emph{Claim 1:The ideal $I$ is cofinite if and only if $I\cap \k[g,g^{-1}]\neq \{0\}.$}\\[1mm]
\emph{Proof of Claim 1. } At first, let $I$ be cofinite. If $I\cap\k[g,g^{-1}]=0$, then $(\k[g,g^{-1}]+I)/I\cong \k[g,g^{-1}]/I\cap \k[g,g^{-1}]=\k[g,g^{-1}]$ which is infinite dimensional. This contradicts to the fact that $I$ is cofinite. Conversely, assume that $I\cap \k[g,g^{-1}]\neq 0$. Since $\k[g,g^{-1}]$ is a principle ideal ring, $I\cap \k[g,g^{-1}]=(p(g))$ for some nonzero polynomial $p(g)\in \k[g]$. Therefore, $(\k[g,g^{-1}]+I)/I\cong \k[g,g^{-1}]/I\cap \k[g,g^{-1}]$ is finite-dimensional which clearly implies that $\k\mathbb{D}_{\infty}/I$ is finite-dimensional too.\qed

Due to this claim, we have $I\cap \k[g,g^{-1}]\neq \{0\}$. As we already shown, there exists a monic polynomial $p(g)\in \k[g]$ such that $I\cap \k[g,g^{-1}]=p(g)\k[g,g^{-1}]$.\\[1.5mm]
\emph{Claim 2: The polynomial $p(g)=\prod_{i}(g-\lambda_i)^{r_i}(g-\lambda_i^{-1})^{r_i}(g-1)^r(g+1)^sg^{t}$ for some $\pm 1\neq\lambda_i\in \k^\ast$ and $r_i,r,s,t\in \mathbb{N}.$}\\[1mm]
\emph{Proof of Claim 2. } By definition, $I$ and thus $(p(g))$ is stable under the action of $G=\Z_2=\langle x\rangle.$ Now $$x\cdot p(g)=xp(g)x^{-1}=p(g^{-1}).$$
This implies that if $\lambda\neq 0$ is a root of $p(g)$, then so is $\lambda^{-1}.$\qed

Based on this claim, we consider a special case at first.\\[1.5mm]
\emph{Claim 3: If $p(g)=(g-\lambda)^{r}(g-\lambda^{-1})^{r}$ for some $\pm 1\neq \lambda\in \k^{\ast}$, then $(\k\mathbb{D}_{\infty}/I)^\ast\subset A.$}\\[1mm]
\emph{Proof of Claim 3. } To show the result, there is no harm to assume that $I$ is the just the ideal of $\k \mathbb{D}_{\infty}$ generated by $p(g).$ Now dim$\k \mathbb{D}_{\infty}/I=4r$ and $\{g^{i}x^{j}|0\leq i\leq 2r-1,0\leq j\leq 1\}$ is a basis. We claim that $$\{E^s\Phi_{\lambda},E^s\Phi_{\lambda^{-1}},E^s\Psi_{\lambda},E^s\Psi_{\lambda^{-1}}|0\leq s\leq r-1\}$$ is a basis of $(\k \mathbb{D}_{\infty}/I)^{\ast}.$ To prove this, we first show that they belong to $(\k \mathbb{D}_{\infty}/I)^{\ast}$ , that is, all of them vanish on $I.$ We use the case $E^s\Phi_{\lambda}$ to explain this fact since the other cases can be proved similarly. In fact, through a straightforward computation one can easily check that
	\begin{align}\label{e6.7}
	&(E^{s}\Phi_{\lambda}, (g-\lambda)^{t})=(E^s\otimes \Phi_{\lambda},\sum_{i=0}^{t}(-\lambda)^{t-i}\left(
	\begin{array}{c}
	t\\
	i
	\end{array}
	\right)
	g^{i}\otimes g^{i})\\
	&=\lambda^{t}\sum_{i=0}^{t}(-1)^{t-i}\left(
	\begin{array}{c}
	t\\
	i
	\end{array}
	\right)i^{s}\notag
	\end{align} which is zero for $s< t$, since it is a Stirling number of the second kind (see (13.13) in \cite{vW}.) Similarly,
	\begin{equation}\label{e6.8}
	(E^{s}\Phi_{\lambda}, (g^{-1}-\lambda^{-1})^t)=0\end{equation}
	for $s<t.$

Now we claim the following fact, that is, \emph{ if for some $f(g)\in \k[g,g^{-1}]$ we have $(E^{s}\Phi_{\lambda}, f(g))=0$ for all $s=0,1,...,t$, then  $(E^{t}\Phi_{\lambda}, f(g)g^k)=0$ for all $k\in\Z$}. To show this fact, there is no harm to assume that $f(g)=\sum_{i=-n}^{n}a_ig^i$ for some $n\in \mathbb{N}.$ The condition implies that
 $$\sum_{i=-n}^{n}a_ii^s\lambda^i\equiv 0$$
 for all $s=0,1,\ldots,t.$ Now direct computation shows that
 \begin{align*}(E^{t}\Phi_{\lambda}, f(g)g^{k})&=\sum_{i=-n}^{n}a_i(i+k)^t\lambda^{i+k}=
 \sum_{i=-n}^{n}\sum_{s=0}^{t}a_i\left(
                                   \begin{array}{c}
                                     t \\
                                     s \\
                                   \end{array}
                                 \right)
 i^sk^{t-s}\lambda^{i+k}\\
 &=\sum_{s=0}^{t} k^{t-s}\lambda^{k} \left(
                                   \begin{array}{c}
                                     t \\
                                     s \\
                                   \end{array}
                                 \right)\sum_{i=-n}^{n}a_ii^{s}\lambda^i\\
                                 &=0.\end{align*}
 This fact implies that $(E^{t}\Phi_{\lambda}, f(g)h(g))=0$ for all $h(g)\in\k[g,g^{-1}]$. Using the same argument, we also have $(E^{t}\Phi_{\lambda}, f(g)h(g)x)=0$ for all $h(g)\in\k[g,g^{-1}]$.
	As elements in the ideal $((g-\lambda)^r(g-\lambda^{-1})^r)$ are linear combinations of $g^k(g-\lambda)^r(g-\lambda^{-1})^rx^\ell$ and $g^i(g^{-1}-\lambda)^r(g^{-1}-\lambda^{-1})^rx^j$ for $k,\ell,i,j\in\Z$,	
	we get
	$$(E^{s}\Phi_{\lambda},((g-\lambda)^r(g-\lambda^{-1})^r))\equiv 0,\;\;\textrm{for}\;s<r$$
by combining equations \eqref{e6.7}, \eqref{e6.8} together with above discussion.

	Similarly, one can show this for $E^s\Phi_{\lambda^{-1}},E^s\Psi_{\lambda},E^s\Psi_{\lambda^{-1}}$ and thus we get that	$$\{E^s\Phi_{\lambda},E^s\Phi_{\lambda^{-1}},E^s\Psi_{\lambda},E^s\Psi_{\lambda^{-1}}|0\leq s\leq r-1\}\subset (\k \mathbb{D}_{\infty}/I)^{\ast}.$$
	By the definition of these elements, we have
	\begin{align}&E^s\Phi_{\lambda}=\sum_{i\in \Z}i^s(f_{i,0}+f_{i,1})\sum_{j\in\Z}
	\lambda^{j}(f_{j,0}+f_{j,1})=\sum_{i\in \Z}i^s\lambda^{i}(f_{i,0}+f_{i,1}),\notag\\
	&E^s\Phi_{\lambda^{-1}}=\sum_{i\in \Z}i^s(f_{i,0}+f_{i,1})\sum_{j\in\Z}
	\lambda^{-j}(f_{j,0}+f_{j,1})=\sum_{i\in \Z}i^s\lambda^{-i}(f_{i,0}+f_{i,1}),\notag\\
	&E^s\Psi_{\lambda}=\sum_{i\in \Z}i^s(f_{i,0}+f_{i,1})\sum_{j\in\Z}
	\lambda^{j}(f_{j,0}-f_{j,1})=\sum_{i\in \Z}i^s\lambda^{i}(f_{i,0}-f_{i,1}),\label{e6.9}\\
	&E^s\Psi_{\lambda^{-1}}=\sum_{i\in \Z}i^s(f_{i,0}+f_{i,1})\sum_{j\in\Z}
	\lambda^{-j}(f_{j,0}-f_{j,1})=\sum_{i\in \Z}i^s\lambda^{-i}(f_{i,0}-f_{i,1}).\notag
	\end{align}
	To show that $\{E^s\Phi_{\lambda},E^s\Phi_{\lambda^{-1}},E^s\Psi_{\lambda},E^s\Psi_{\lambda^{-1}}|0\leq s\leq r-1\}$ is a basis of $(\k \mathbb{D}_{\infty}/I)^{\ast}$, it is enough to show that they are linear independent by noting that dim$(\k \mathbb{D}_{\infty}/I)^{\ast}=4r.$ From the expression \eqref{e6.9}, one can construct the following elements through linear combinations of $\{E^s\Phi_{\lambda},E^s\Phi_{\lambda^{-1}},E^s\Psi_{\lambda},E^s\Psi_{\lambda^{-1}}|0\leq s\leq r-1\}$:
	\begin{align*}
	& p_{0}:=\sum_{i\in \Z}\lambda^{i}f_{i,0},\;\;p_1:=\sum_{i\in \Z}i\lambda^{i}f_{i,0},\;\;\cdots\;\;p_{r-1}:=\sum_{i\in \Z}i^{r-1}\lambda^{i}f_{i,0},\\
	&q_{0}:=\sum_{i\in \Z}\lambda^{-i}f_{i,0},\;\;q_1:=\sum_{i\in \Z}i\lambda^{-i}f_{i,0},\;\;\cdots\;\;q_{r-1}:=\sum_{i\in \Z}i^{r-1}\lambda^{-i}f_{i,0},\\
	& p'_{0}:=\sum_{i\in \Z}\lambda^{i}f_{i,1},\;\;p'_1:=\sum_{i\in \Z}i\lambda^{i}f_{i,1},\;\;\cdots\;\;p'_{r-1}:=\sum_{i\in \Z}i^{r-1}\lambda^{i}f_{i,1},\\
	&q'_{0}:=\sum_{i\in \Z}\lambda^{-i}f_{i,1},\;\;q'_1:=\sum_{i\in \Z}i\lambda^{-i}f_{i,1},\;\;\cdots\;\;q'_{r-1}:=\sum_{i\in \Z}i^{r-1}\lambda^{-i}f_{i,1}.
	\end{align*}
	So it is enough to show $\{p_i,q_i,p'_i,q'_i|0\leq i\leq r-1\}$ are linear independent in $(\k \mathbb{D}_{\infty}/I)^{\ast}$. Clearly, the elements with $'$ and those without $'$ have disjoint support on the basis. Thus we only need to show that $\{p_i,q_i|0\leq i\leq r-1\}$ (resp. $\{p'_i,q'_i|0\leq i\leq r-1\}$) are linear independent. We only show that $\{p_i,q_i|0\leq i\leq r-1\}$ are linear independent here since one can prove the other case in the same way. Now assume that	\begin{equation}\label{e6.10}\sum_{i=0}^{r-1}k_ip_i+\sum_{i=0}^{r-1}l_iq_i=0\end{equation}
	for some $k_i,l_i\in \k.$ Let equation \eqref{e6.10} act on $\{g^{j}|0\leq j\leq 2r-1\}$ respectively, then we get a system of homogeneous linear equations of $k_i,l_i$. The coefficient matrix of this system is
	$$B:=\left(
	\begin{array}{cccccccc}
	1 & 0 & \cdots& 0& 1 & 0 & \cdots & 0 \\
	\lambda & \lambda & \cdots & \lambda & \lambda^{-1} & \lambda^{-1} & \cdots & \lambda^{-1} \\
	\lambda^2& 2\lambda^2 & \cdots & 2^{r-1}\lambda^2 & \lambda^{-2} & 2\lambda^{-2}& \cdots &2^{r-1} \lambda^{-2}\\
	\lambda^3& 3\lambda^3 & \cdots & 3^{r-1}\lambda^3 & \lambda^{-3} & 3\lambda^{-3}& \cdots &3^{r-1} \lambda^{-3} \\
	\vdots& \vdots & \cdots & \vdots & \vdots & \vdots & \cdots & \vdots \\
	\lambda^{r'}& r'\lambda^{r'} & \cdots & r'^{r-1}\lambda^{r'} & \lambda^{-r'} & r'\lambda^{-r'}& \cdots &r'^{r-1} \lambda^{-r'}
	\end{array}
	\right),$$
	for $r'=2r-1.$
	We need to show that the determinant $|B|\neq 0.$ By Corollary~\ref{cor}, this is indeed the case. Therefore, $\{p_i,q_i|0\leq i\leq r-1\}$ are linear independent.\qed

Now we can finish the proof. At first, one can repeat above proof (which will be easier) to show Claim 3 in the following cases: 1) $p(g)=(g-1)^s$ for some $s\in \mathbb{N}$; 2) $p(g)=(g+1)^s$ for some $s\in \mathbb{N}$; 3) $p(g)=g^{s}$ for some $s\in \mathbb{N}$ (in this case, we can assume that $p(g)$ is just $1$ of course). Therefore, for the general $p(g)$, one can use the Chinese Remainder Theorem to show that $(\k\mathbb{D}_{\infty})^{\circ}$ is just the subalgebra $A$.
\end{proof}

\subsection{Proof of Theorem \ref{thm dual}.} We are in the position to give the proof of Theorem \ref{thm dual} now. For this, we define the following map
$$\Theta\colon \k \mathbb{D}_{\infty^{\circ}}\to (\k \mathbb{D}_{\infty})^{\circ},\;\;
F\mapsto E,\;\phi_{\lambda}\mapsto \Phi_{\lambda},\;\psi_{\lambda}\mapsto \Psi_{\lambda},\;\;\;(\lambda\in \k^{\ast}).$$	
According to equations \eqref{e6.2}-\eqref{e6.4}, this map $\Theta$ extends to an algebra morphism naturally which is still denoted by $\Theta$. In addition, it is a Hopf morphism by Lemma \ref{l11}.

This Hopf morphism is injective since $\{F^i\phi_{\lambda},F^i\psi_{\lambda}|i\in \mathbb{N},\lambda\in \k^{\ast}\}$ is a basis of $\k \mathbb{D}_{\infty^{\circ}}$ (by Diamond Lemma \cite{Ber}) and $\{E^i\Phi_{\lambda},E^i\Psi_{\lambda}|i\in \mathbb{N},\lambda\in \k^{\ast}\}$ are linear independent in $\k \mathbb{D}_{\infty}^{\circ}$ (to show this, take finite number of them. Then, as the proof of Claim 3 in Proposition \ref{prop12} shows that they are already linear independent in $(\k\mathbb{D}_{\infty}/I)^{\ast}$ for some cofinite ideal $I\subset \k\mathbb{D}$).

Proposition \ref{prop12} tells us that $\Theta$ is surjective too. Combining all of these statements, $\Theta$ is a Hopf isomorphism.\qed



\begin{thebibliography}{99}
\bibitem{BGR}  E. Basor, F. Ge, M.O. Rubinstein, \textit{Some multidimensional integrals in number theory and connections with the Painleve V equation}, J. Math. Phys. 59, 091404 (2018).

\bibitem{Ber} G.M. Bergman, \textit{The diamond lemma for ring theory}, Adv. in Math. 29 (1978), no. 2, 178-218.

\bibitem{BC} S. Bettin, J.B. Conrey,\textit{ Averages of long Dirichlet polynomials}, arXiv:2002.09466.

\bibitem{BC1} K. A. Brown, M. Couto, \textit{Affine commutative-by-finite Hopf algebras}, arXiv:1907.10527v1.

\bibitem{BZ} K.A. Brown, J.J. Zhang, \textit{Prime regular Hopf algebras of GK-dimension one}, Proc. London Math. Soc. (3) 101 (2010) 260-302.

\bibitem{CM} W. Chin, I. M. Musson, \textit{Hopf algebra duality, injective modules and quantum groups}, Comm. Algebra 22 (1994), no. 12, 4661-4692.

\bibitem{DHI} A. Dea{\~n}o, D. Huybrechs, A. Iserles, \textit{The kissing polynomials and their Hankel determinants}, arXiv:1504.07297.

\bibitem{Ful}	W. Fulton, \textit{Young tableaux}, Cambridge University Press, Cambridge, 1997.
	
\bibitem{KRRR} J.P. Keating, B. Rodgers, E. Roditty-Gershon, Z. Rudnick, \textit{Sums of divisor functions in $\mathbb F_q[t]$ and matrix integrals}, Math. Z. (2018) 288: 167-198.

\bibitem{Kra} C. Krattenthaler, \textit{Growth diagrams, and increasing and decreasing chains in fillings of Ferrers shapes}, Advances in Applied Mathematics 37(3):404-431.

\bibitem{L} G. Liu, \textit{The dual of Hopf algebras of GK-1}, preprint.

\bibitem{L1} G. Liu, \textit{A classification result on prime Hopf algebras of GK-dimension one},  J. Algebra   547  (2020),  579-667.

\bibitem{LWZ} D.-M. Lu, Q.-S. Wu and J. J. Zhang, \textit{Homological integral of Hopf algebras}, Trans. Amer. Math. Soc. 359 (2007), 4945-4975.

\bibitem{MM} J. W. Milnor, J. C. Moore, \textit{On the structure of Hopf algebras},
Ann. of Math. (2) 81 (1965), 211-264.

\bibitem{Mo} S. Montgomery, \textit{Hopf algebra and their actions on rings}, CBMS Regional Conference Series in Mathematics, 82, Providence, RI, 1993.

\bibitem{RS} B. Rodgers, K. Soundararajan, \textit{The variance of divisor sums in arithmetic progressions}, Forum Math. 30 (2018), no. 2, 269-293.

\bibitem{Ta} M. Takeuchi, \textit{Hopf algebra techniques applied to the quantum group $U_{q}(sl(2))$}, Deformation theory and quantum groups with applications to mathematical physics (Amherst, MA, 1990), 309-323, Contemp. Math., 134, Amer. Math. Soc., Providence, RI, 1992.

\bibitem{vW} van Lint, J. H.; Wilson, R. M. \textit{A course in combinatorics}, Second edition. Cambridge University Press, Cambridge, 2001. xiv+602 pp. ISBN: 0-521-00601-5 05-01 (90B10)

\bibitem{WLD}  J. Wu, G. Liu, N. Ding, \textit{Classification of affine prime regular Hopf algebras of GK-dimension one}, Adv. Math. 296 (2016), 1-54.
\end{thebibliography}
\end{document}